\documentclass[11pt]{amsproc}
\usepackage{amsmath,enumitem}
\usepackage{amssymb,amsthm}
\usepackage{graphics}
\usepackage{array}
\usepackage{fullpage}

\usepackage{color, url}
\usepackage{float}
\usepackage{graphicx}
\usepackage{hyperref}
\usepackage{comment} 

\usepackage{here} 

\usepackage[dvipsnames]{xcolor}
\usepackage[capitalize,nameinlink,nosort]{cleveref}
\hypersetup{
	pdftoolbar=true,        
	pdfmenubar=true,        
	pdffitwindow=false,     
	pdfstartview={FitH},    
	pdftitle={},
	pdfauthor={},
	pdfsubject={},
	pdfkeywords={},
	pdfnewwindow=true,      
	colorlinks=true,       
	linkcolor=brown,          
	citecolor=brown,        
	filecolor=brown,      
	urlcolor=brown,           
	unicode=true          
}

\makeatletter
\@namedef{subjclassname@2020}{%
  \textup{2020} Mathematics Subject Classification}
\makeatother

\newtheorem{theoremcounter}{Theorem Counter}[section]

\theoremstyle{definition}
\newtheorem{definition}[theoremcounter]{Definition}

\newtheorem{example}[theoremcounter]{Example}

\theoremstyle{plain}
\newtheorem{lemma}[theoremcounter]{Lemma}
\newtheorem{proposition}[theoremcounter]{Proposition}
\newtheorem{corollary}[theoremcounter]{Corollary}

\newtheorem{theorem}[theoremcounter]{Theorem}

\numberwithin{equation}{section}

\newcommand{\Z}{{\mathbb Z}}
\def\li{\text{\rm Li}}

\def\sts#1#2{\genfrac{\{}{\}}{0pt}{}{#1}{#2}}
\def\stf#1#2{\genfrac{[}{]}{0pt}{}{#1}{#2}}
\def\eal#1#2{\genfrac{}{}{0pt}{}{#1}{#2}}

\def\k#1{{#1}_B} 
\def\p#1{{#1}_R} 

\newcommand{\red}{\text{red} } 
\newcommand{\blue}{\text{blue} } 

\title[Combinatorial aspects of poly-Bernoulli polynomials]{Combinatorial aspects of poly-Bernoulli polynomials and poly-Euler numbers}

\author{Be\'ata B\'enyi}
\address{\noindent Faculty of Water Sciences, University of Public Service, Baja, HUNGARY}
\email{benyi.beata@uni-nke.hu}

\author{Toshiki Matsusaka}
\address{Institute for Advanced Research, Nagoya University, Furo-cho, Chikusa-ku, Nagoya, 464-8601, JAPAN}
\email{matsusaka.toshiki@math.nagoya-u.ac.jp}

\date{\today}
\thanks{The second author was supported by JSPS KAKENHI Grant Number 20K14292.}

\begin{document}


\begin{abstract}
	In this article, we introduce combinatorial models for poly-Bernoulli polynomials and poly-Euler numbers of both kinds. As their applications, we provide combinatorial proofs of some identities involving poly-Bernoulli polynomials.
\end{abstract}

\maketitle

\section{Introduction}

The literature on poly-Bernoulli numbers and related topics is developing nowadays more and more. Kaneko \cite{Kaneko1997} observed in 1997 that the generating function of the Bernoulli numbers can be generalized using the polylogarithm function and introduced \emph{poly-Bernoulli numbers}. Arakawa and Kaneko \cite{ArakawaKaneko1999_2} showed the strong connection of the newly introduced numbers to multiple zeta values.
These results motivated numerous studies and generalizations in different directions. Not only directly generalizations of poly-Bernoulli numbers were studied (poly-Bernoulli polynomials, multi-poly-Bernoulli polynomials, symmetrized poly-Bernoulli polynomials, etc.) but also the ``poly"-versions of different well-known number sequences and polynomials were introduced and investigated as, for instance, poly-Cauchy numbers, poly-Euler numbers, etc. 

The amazing fact about poly-Bernoulli numbers is that though the introduction was motivated by a simple analytical fact (it is ``nice") it turned out that the poly-Bernoulli numbers with negative indices play role in many combinatorial aspects. 
Poly-Bernoulli numbers and their relatives arise in different research fields, such as discrete tomography, mathematics of origami, biology, etc.~\cite{Brewbaker2008, HostenSullivant2010, LangHowell2018, LetsouCai2016, Ryser1957}. 
Numerous combinatorial interpretations are known, lonesum matrices, $\Gamma$-free matrices, alternative tableaux of rectangular shapes, Le-tableaux of rectangular shape, tree-like tableaux of rectangular shapes, Vesztergombi permutations, 
Callan permutations, permutations with excedance set containing the initial values, permutations that can be suffix arrays, acyclic orientations of complete bipartite graphs, etc. See \cite{BenyiHajnal2017} and the references therein. The frequent appearance of these numbers underlines their importance.

This article is organized as follows. First, we extend the notion of Callan permutations (Callan sequences) which are one of the first combinatorial interpretations of poly-Bernoulli numbers. 
As applications, we provide combinatorial proofs for some identities as, for instance, the relation between poly-Bernoulli polynomials and (hyper-)sums of powers of positive integers. 

Second, this extended model is also fundamental in the combinatorial model that we present for the poly-Euler numbers of the first kind introduced by Ohno and Sasaki \cite{OhnoSasaki2012} and poly-Euler numbers of the second kind recently introduced by Komatsu \cite{Komatsu2017,Komatsu2020}. 
We prove a number of the known identities combinatorially and give a combinatorial explanation for the relation between poly-Euler numbers and poly-Bernoulli numbers.

\section{Poly-Bernoulli polynomials} \label{s2}

Throughout this article, we refer to an element with the subscript $R$ (resp. $B$) as a ``red" (resp. ``blue") element. Poly-Bernoulli polynomials were defined in different ways in the literature~\cite{BayadHamahata2011, CoppoCandelpergher2010, KomatsuLuca2013}. 
We consider in this article the definition of Bayad and Hamahata~\cite{BayadHamahata2011}. 

\begin{definition} For every integer $k$, the  polynomials $(B_n^{(k)}(x))_{n = 0}^\infty$ are called \emph{poly-Bernoulli polynomials} and defined by
	\begin{align*}
		\sum_{n=0}^{\infty}B_n^{(k)}(x)\frac{t^n}{n!}=\frac{\li_k(1-e^{-t})}{1-e^{-t}}e^{xt},
	\end{align*}
where $\li_k(z) =\sum_{m=1}^{\infty} z^m/m^k$ is the $k$-th polylogarithm function.
\end{definition}
The numbers $B_n^{(k)}:= B_n^{(k)}(0)$ are the poly-Bernoulli numbers, introduced by Kaneko \cite{Kaneko1997} and studied since then extensively.
The classical Bernoulli polynomials are  $B_n(x) = (-1)^n B_n^{(1)}(-x)$.

We describe two ways to interpret these polynomials based on two generalizations of Stirling numbers of the second kind. 
In the first model, $x$ is assumed to be a positive integer, while in the second, $x$ marks a weight, hence, it is a formal variable, being arbitrary. These two different approaches are general in combinatorics. 
One gives a very concrete picture while the other allows more generality, greater freedom of abstraction. Both approaches have their own advantages.

Let us recall now the definition of Callan sequences that are enumerated by the poly-Bernoulli numbers with negative indices.  

Let $N=\{\p{1},\ldots, \p{n}\}\cup\{\p{*}\}$ (referred to as \red elements) and  $K=\{\k{1},\ldots, \k{k}\}\cup\{\k{*}\}$ (referred to as \blue elements). Let $R_1,\ldots, R_m,R^*$ be an ordered partition of the set $N$ into $m+1$ non-empty blocks ($0 \leq m \leq n$) and $B_1,\ldots,B_m,B^*$ an ordered partition of the set $K$ into $m+1$ non-empty blocks, where $R^*$ denotes the block that contains $\p{*}$ and $B^*$ denotes the block that contains $\k{*}$. 
We call $B^*$ and $R^*$ \emph{extra blocks} while the other blocks \emph{ordinary blocks}. A pair $(B_i; R_i)$ is called an \emph{ordinary Callan pair} and $(B^*; R^*)$ an \emph{extra Callan pair}. 
A \emph{Callan sequence of size $n \times k$} (or an \emph{$(n,k)$-Callan sequence}) is a linear arrangement of Callan pairs augmented by the extra pair 
\[(B_1; R_1)(B_2; R_2)\cdots(B_m; R_m)\cup(B^*; R^*).\]  
It is known that the poly-Bernoulli number $B_n^{(-k)} = B_n^{(-k)}(0)$ counts Callan sequences of size $n \times k$, (see~\cite[Theorem 5]{BenyiMatsusaka2021}).
 
\subsection{Extended Callan sequences}\label{subsec: extended_Callan}

This generalization is motivated by the so-called $r$-Stirling numbers studied first by Broder~\cite{Broder1984}. 
The \emph{$r$-Stirling numbers of the second kind}, $\sts{n}{k}_r$ count partitions of $[n+r]=\{1,2,\ldots, n+r\}$ into $k+r$ non-empty blocks such that the elements $\{1,2,\ldots, r\}$ are in distinct blocks. The generating function is given by
\begin{align*}
	\sum_{n=0}^{\infty} \sts{n}{k}_r \frac{t^n}{n!} =e^{rt}\frac{(e^t-1)^k}{k!}.
\end{align*}

We define an \emph{$r$-extended Callan sequence} in the following way. We augment the set $N$ with $r$ new \emph{special} \red elements $N' = N\cup\{\p{(n+1)},\p{(n+2)},\ldots, \p{(n+r)}\}$ and partition the set $N'$ into $m+r+1$ blocks such that the special elements and the extra element $\p{*}$ are in distinct blocks. 
We partition the set of \blue elements $K$ into $m+1$ blocks. As before, the blocks containing $\k{*}$ and $\p{*}$ create the extra Callan pair. 
The remaining blocks of \blue elements and the blocks of \red elements not containing any special elements form the ordinary Callan pairs. 

An $r$-extended Callan sequence of size $n \times k$ is a linear arrangement of ordinary Callan pairs augmented by the extra pair and the special blocks
\[
(B_1; R_1) \cdots (B_m; R_m) \cup (B^*; R^*) \cup (S_1, \p{(n+1)}) \cup \cdots \cup (S_r, \p{(n+r)})
\]
with $0 \leq m \leq n$, where $S_1,\ldots, S_r$ contain \red elements and can be empty.

\begin{example}A $3$-extended Callan sequence of size $9\times 6$ is 
 \begin{align*}(\k{5}; \p{6},\p{8})(\k{3},\k{6}; \p{4},\p{9})(\k{1},\k{2},\k{4}; \p{3}) \cup (\k{*}; \p{1}, \p{*}) \cup (\p{5},\p{\bf{10}}) \cup (\p{\bf{11}}) \cup (\p{2},\p{7},\p{\bf{12}}).
 	\end{align*}
 The special elements $\p{\bf{10}}$, $\p{\bf{11}}$, $\p{\bf{12}}$ are denoted boldface.
\end{example}

\begin{theorem}\label{theo: pBp_1}
	Let $\mathcal{EC}(n,k;r)$ denote the set of $r$-extended Callan sequences of size $n\times k$. Then we have
	\begin{align*}
		B_n^{(-k)}(r) = |\mathcal{EC}(n,k;r)|.
	\end{align*}
\end{theorem}

\begin{proof}
	We can construct the blocks $S_1, \dots, S_r$ of an $r$-extended Callan sequence in ${n \choose j} r^j$ ways: choose $j$ \red elements in $\binom{n}{j}$ ways and choose for each chosen element a special block. 
	The remaining $n-j$ \red elements and $k$ \blue elements form Callan sequences in $B_{n-j}^{(-k)}(0)$ ways. 
	Thus, we have
	\begin{align}\label{Bernoulli-poly-exp}
		|\mathcal{EC}(n,k; r)| = \sum_{j=0}^n {n \choose j} B_{n-j}^{(-k)}(0) r^j,
	\end{align}
	which coincides with the formula for $B_n^{(-k)}(r)$ given in~\cite[Remark 3]{CoppoCandelpergher2010}.
\end{proof}

The combinatorial proofs of the following theorems easily follow.  

\begin{theorem}\cite{KarginCenkciDilCan2020}\label{theo:pBpol_closed form}	
	For non-negative integers $n$, $k$, and $r$, we have
	\begin{align}\label{def_comb}
		B_n^{(-k)}(r)=\sum_{m=0}^{\min(n,k)}(m!)^2\sts{k+1}{m+1}\sts{n}{m}_{r+1}.
	\end{align}
\end{theorem}

\begin{proof}
	Let $m$ be the number of ordinary Callan pairs in the $r$-extended Callan sequence. There are $\sts{n}{m}_{r+1}$ ways to partition $N'$ into $m+r+1$ non-empty blocks such that the special elements and $\p{*}$ are in distinct blocks. 
	There are $\sts{k+1}{m+1}$ ways to partition $K$ into $m+1$ non-empty blocks. We arrange the blocks of \red and \blue elements that do not contain any special elements nor extra elements in $(m!)^2$ ways.
\end{proof}

\begin{theorem}\cite{KarginCenkciDilCan2020}\label{Bnkr-anti}
For $n, k, r \geq 0$, it holds
\begin{align}\label{pBpol_sieve}
	B_n^{(-k)}(r)=\sum_{m=0}^{k}(-1)^{k-m}\sts{k}{m}m!(m+r+1)^n.
\end{align}
\end{theorem}

\begin{proof}
	We prove it by the inclusion-exclusion principle (see~\cite[Section 2.1]{Stanley1997}). Our objects here are permutations
	\begin{align}\label{permu-Thm2.5}
		(R_0 B_1 R_1 B_2 R_2 \cdots B_j R_j; S_1; \cdots; S_r) \quad (1 \leq j \leq k)
	\end{align}
	of $\{\p{1}, \dots, \p{n}\}$ and $\{\k{1}, \dots, \k{k}\}$ constructed as follows.
	\begin{enumerate}
		\item $S_1, \dots, S_r$ are (possible empty) blocks of \red elements. In each block, all \red elements are in increasing order.
		\item The remaining \red and \blue elements form a permutation in which each substring of \red elements is increasing. Note that the blocks $R_0$ and $R_j$ can be empty.
	\end{enumerate}
	If all \blue elements in each $B_i$ are also in increasing order, then we get an $r$-extended Callan sequence as
	\begin{align*}
	\begin{cases}
		(B_1; R_1) \cdots (B_j; R_j) \cup (R_0, \p{*}; \k{*}) \cup (S_1, \p{(n+1)}) \cup \cdots \cup (S_r, \p{(n+r)}) \quad \text{if } R_j \neq \emptyset,\\
		(B_1; R_1) \cdots (B_{j-1}; R_{j-1}) \cup (R_0, \p{*}; B_j, \k{*}) \cup (S_1, \p{(n+1)}) \cup \cdots \cup (S_r, \p{(n+r)}) \quad \text{if } R_j = \emptyset.
	\end{cases}
	\end{align*}
	
	To count the number of such permutations, we recall the notion of the \emph{descent set} $D(p)$. 
	For a permutation $p = p_1 p_2 \cdots p_n$, the entry $p_i$ is called a \emph{descent} of $p$ if $p_i > p_{i+1}$. We let $D(p)$ denote the set of all descents of $p$.
	
	For a subset $A$ of $\{\k{2}, \dots, \k{k}\}$, let $\alpha(A)$ denote the number of permutations in \eqref{permu-Thm2.5} with the condition $D(B_1) \cup \cdots \cup D(B_j) = A$. 
	Let $\beta(A)$ denote the number of permutations with $D(B_1) \cup \cdots \cup D(B_j) \supset A$. Then $\alpha(\emptyset)$ counts the number of $r$-extended Callan sequences, that is, $\alpha(\emptyset) = B_n^{(-k)}(r)$. 
	
	For $k=0$ and $k=1$, we easily see that
	\[
		\alpha(\emptyset) = \begin{cases}
			(r+1)^n &\text{if } k=0,\\
			(r+2)^n &\text{if } k=1.
		\end{cases}
	\]
	
	For $k \geq 2$, by the inclusion-exclusion principle~\cite[Theorem 2.1.1]{Stanley1997}, we have
	\[
		\alpha(\emptyset) = \sum_{A \subset \{\k{2}, \dots, \k{k}\}} (-1)^{|A|} \beta (A) = \sum_{m=0}^{k-1} (-1)^m \sum_{\substack{A \subset \{\k{2}, \dots, \k{k}\}\\ |A| = m}} \beta (A).
	\]
	Now the inner sum is given as follows. First, we construct an ordered partition of $\{\k{1}, \dots, \k{k}\}$ with $(k-m)$-non-empty blocks in $\sts{k}{k-m} (k-m)!$ ways. In each block, we record the elements in decreasing order. 
	Next, to obtain a permutation in \eqref{permu-Thm2.5}, we insert \red elements of $\{\p{1},\dots, \p{n}\}$ before and after the blocks of \blue elements and into the special blocks $S_1, \dots, S_r$ in $(k-m+1+r)^n$ ways. 
	If two blocks of \blue elements are consecutive, then we connect them. The resulting permutations have at least $m$ descents in blocks of \blue elements.
	
	Therefore we have
	\begin{align*}
		\alpha(\emptyset) &= \sum_{m=0}^{k-1} (-1)^m \sts{k}{k-m} (k-m)! (k-m+1+r)^n,
	\end{align*}
	which concludes the proof.
\end{proof}

\begin{example}
	The number $\beta(\emptyset) = k! (k+1+r)^n$ counts all permutations in \eqref{permu-Thm2.5}. Let $(n,k,r) = (1,3,0)$ for instance. 
	In the above proof, the permutation $\k{3} \k{2} \k{1} \p{1}$ is counted as $(\k{3}) (\k{2}) (\k{1}) \p{1}$ with $m=0$, $(\k{3} \k{2}) (\k{1}) \p{1}$ and $(\k{3}) (\k{2} \k{1}) \p{1}$ with $m=1$, and $(\k{3} \k{2} \k{1}) \p{1}$ with $m=2$.
\end{example}

The symmetry of the roles of the sets $N$ and $K$ in Callan sequences is lost in the model of $r$-extended Callan sequences.
We define now the $(r,s)$-extended Callan sequences to restore this symmetry in some sense.

\begin{definition}
 	An \emph{$(r,s)$-extended Callan sequence of size $n \times k$} is a linear arrangement of Callan pairs augmented by the extra pair, $r$ special blocks of \red elements, and $s$ special blocks of \blue elements, that is,
	\begin{align*}
		&(B_1; R_1) \cdots (B_m; R_m) \cup (B^*; R^*)\\
		&\cup (S_1, \p{(n+1)}) \cup \cdots (S_r, \p{(n+r)}) \cup (C_1, \k{(k+1)}) \cup \cdots \cup (C_s, \k{(k+s)}).
	\end{align*}
	Here $R_1, \dots, R_m, R^*, S_1, \dots, S_r$ (resp. $B_1, \dots, B_m, B^*, C_1, \dots, C_s$) is an ordered partition of the set $N$ (resp. $K$) with possible empty blocks $S_1, \dots, S_r$ (resp. $C_1, \dots, C_s$). 
\end{definition}

Bayad--Hamahata~\cite{BayadHamahata2011} defined the following function:
\begin{align*}
	C_n^{(-k)}(x,y) = \sum_{j=0}^{k}\binom{k}{j}B_n^{(-j)}(x)y^{k-j}.
\end{align*}
\begin{theorem} The number of $(r,s)$-extended Callan sequences of size $n\times k$ is $C_n^{(-k)}(r,s)$.
\end{theorem}
\begin{proof}
The proof is analogous to that of \cref{theo: pBp_1}.
\end{proof}

\begin{theorem}[Duality~\cite{BayadHamahata2011}]
	We have 
	\begin{align*}
		C_n^{(-k)}(r,s)  = C_k^{(-n)}(s,r).
	\end{align*}
\end{theorem}

\begin{proof}
By exchanging the role of \red and \blue elements, we obtain that the size of the set of $(r,s)$-extended Callan sequences of size $n\times k$ is the same as the set of $(s,r)$-extended Callan sequences of size $k\times n$. 
\end{proof}

\begin{theorem}[Inversion formula~\cite{BayadHamahata2011}]
	For $n,k,r \geq 0$, we have
	\begin{align*}
		B_n^{(-k)}(r) = \sum_{j=0}^{k} (-1)^j \binom{k}{j} C_n^{(-k+j)}(r,s) s^j.
	\end{align*}
\end{theorem}

\begin{proof}
	We use the inclusion-exclusion principle similar to the proof of \cref{Bnkr-anti}. For any $A \subset \{\k{1}, \dots, \k{k}\}$, let $\alpha(A)$ (resp. $\beta (A)$) denote the number of $(r,s)$-extended Callan sequences with special blocks of \blue elements $C_1 \cup \cdots \cup C_s = A$ (resp. $C_1 \cup \cdots \cup C_s \supset A$). 
	Then the inclusion-exclusion principle implies that
	\[
		B_n^{(-k)}(r) = \alpha(\emptyset) = \sum_{A \subset \{\k{1}, \dots, \k{k}\}} (-1)^{|A|} \beta (A).
	\]
	For each $A$, we construct $(r,s)$-extended Callan sequences such that $A \subset C_1 \cup \cdots \cup C_s$ as follows. 
	By our assumption, all elements of $A$ are in $C_1 \cup \cdots \cup C_s$ in $s^{|A|}$ ways. The remaining elements form an $(r,s)$-extended Callan sequences of size $n \times (k-|A|)$. Thus we have
	\[
		\beta (A) = C_n^{(-k+|A|)}(r,s) s^{|A|},
	\]
	which implies the assertion.
\end{proof}

\begin{theorem}[Closed formula~\cite{BayadHamahata2011}]
	For $n,k,r \geq 0$, we have
	\begin{align*}
		C_n^{(-k)} (r,s) = \sum_{m=0}^{\min(n,k)} (m!)^2 \left(\sum_{i=0}^{n}(r+1)^{n-i}\binom{n}{i}\sts{i}{m}\right)\left( \sum_{j=0}^{k}(s+1)^{k-j}\binom{k}{j}\sts{j}{m}\right).
	\end{align*}
\end{theorem}
\begin{proof}
	We enumerate $(r,s)$-extended Callan sequences. Let $m$ be the number of ordinary block pairs, $i$ the number of \red elements and $j$ the number of \blue elements contained in the ordinary blocks. 
	Choose from the \red elements in $\binom{n}{i}$ ways and the \blue elements in $\binom{k}{j}$ ways these elements and construct the sequence of the ordinary pairs in $(m!)^2\sts{i}{m}\sts{j}{m}$ ways. 
	Additionally, we create the (possible empty) extra and special blocks from the \red and \blue elements in $(r+1)^{n-i}$ and $(s+1)^{k-j}$ ways respectively. 
\end{proof}
\begin{theorem}\label{theo:sym_pB}
	For non-negative integers $n$, $k$, $r$, and $s$, it holds
	\begin{align*}
		C_n^{(-k)}(r,s)  = \sum_{m=0}^{\min(n,k)} (m!)^2\sts{n}{m}_{r+1}\sts{k}{m}_{s+1}.
	\end{align*}
\end{theorem}
\begin{proof}
	The proof is analogous to that of \cref{theo:pBpol_closed form}.
\end{proof}

\subsection{Abundant Callan sequences}\label{subsec: abundant Callan}
The $r$-extended Callan sequences are defined for non-negative integers $r$. 
Next, we introduce a combinatorial model for the poly-Bernoulli polynomials with an indeterminate $x$.

For this sake, we consider Carlitz's weighted Stirling numbers~\cite{Carlitz1980-1, Carlitz1980-2} instead of $r$-Stirling numbers.

\begin{definition}
	For non-negative integers $n$ and $m$, let $\mathcal{S}_2(n,m)$ denote the set of partitions of $\{1, \dots, n\} \cup \{\dagger\}$ into $m+1$ non-empty blocks. Let $P^\dagger$ denote the block containing $\dagger$. 
	Let $w_{\mathcal{S}_2}: \mathcal{S}_2(n,m) \to \Z_{\geq 0}$ denote the weight $w_{\mathcal{S}_2}(p) = |P^\dagger|-1$. We define the \emph{Stirling polynomial of the second kind} by
	\[
		\sts{n}{m}_x = \sum_{p \in \mathcal{S}_2(n,m)} x^{w_{\mathcal{S}_2}(p)}.
	\]
\end{definition}

Note that $\sts{n}{0}_x = x^n$ for  $n \geq 0$ and $\sts{n}{m}_x = 0$ for $m > n$. 
Moreover, for  $x = r \in \Z_{\geq 0}$, this polynomial coincides with the $r$-Stirling number of the second kind.
\begin{lemma}\label{rec-stirling-sec}
	For $n,m \geq 0$, we have
	\begin{align*}
		\sts{n+1}{m}_x &= (x+m)\sts{n}{m}_x + \sts{n}{m-1}_x,\\
		\sts{n}{m}_{x+1} &= (m+1) \sts{n}{m+1}_x + \sts{n}{m}_x.
	\end{align*}
\end{lemma}

\begin{proof}
	We define a map $f: \mathcal{S}_2(n+1, m) \to \mathcal{S}_2(n,m)$ by deleting $n+1$. For each $p \in \mathcal{S}_2(n+1, m)$, the following three cases occur.
	\begin{enumerate}
		\item If $n+1$ is in the special block $P^\dagger$, then the number of blocks is preserved and the weight is decreased by one.
		\item If $n+1$ is in an ordinary block with other elements, then the number of blocks and the weight are preserved.
		\item If $n+1$ is alone in a block, then the number of blocks is decreased by one and the weight is preserved.
	\end{enumerate}
	Therefore, we have
	\[
		\sts{n+1}{m}_x = x \sts{n}{m}_x + m \sts{n}{m}_x + \sts{n}{m-1}_x,
	\]
	which implies the first assertion.
	
	The recurrence relation is translated into the generating function
	\[
		\frac{(e^t-1)^m}{m!} e^{tx} = \sum_{n=m}^\infty \sts{n}{m}_x \frac{t^n}{n!}.
	\]
	The equation
	\[
		\frac{(e^t-1)^m}{m!} e^{t(x+1)} = \frac{(e^t-1)^m}{m!} e^{tx} ((e^t-1)+1)
	\]
	implies the second assertion.
\end{proof}

We define \emph{abundant Callan sequences}.
Let $N = \{\p{1}, \dots, \p{n}\} \cup \{\p{*}\}$ denote as before the set of $n+1$ \red elements and $K = \{\k{1}, \dots, \k{k}\} \cup \{\k{*}\}$ the set of $k+1$ \blue elements.
Let $R_1, \dots, R_m, R^*, S$ be an ordered partition of $N$ into $m+2$ blocks ($0 \leq m \leq n$). 
The block containing $\p{*}$ is denoted by $R^*$. We assume that the blocks $R_1, \dots, R_m$ are non-empty and the block $S$ can be empty. 
Similarly, let $B_1, \dots, B_m, B^*$ be an ordered partition of $K$ into $m+1$ non-empty blocks, where the block containing $\k{*}$ is denoted by $B^*$.

\begin{definition}
	An \emph{abundant Callan sequence} of size $n \times k$, or \emph{$(n,k)$-abundant Callan sequence}, is a sequence of ordinary pairs $(B_i, R_i)$ with an extra pair $(B^*, R^*)$ and a special block $S$:
\[
	C = (B_1; R_1) \cdots (B_m; R_m) \cup (B^*; R^*) \cup S.
\]

\end{definition} 

For each abundant Callan sequence $C$, we define the weight by $w_{\mathcal{AC}}(C) = |S|$.

\begin{theorem}
	For $n, k \geq 0$, let $\mathcal{AC}(n,k)$ denote the set of all abundant Callan sequences of size $n \times k$. Then we have
	\[
	\sum_{C \in \mathcal{AC}(n,k)} x^{w_{\mathcal{AC}}(C)} = B_n^{(-k)}(x).
	\]
\end{theorem}

\begin{proof}
	For each $0 \leq j \leq n$, we count the number of elements $C \in \mathcal{AC}(n,k)$ such that $w_\mathcal{AC}(C) = j$. 
	First, the special block of size $j$ can be chosen in ${n \choose j}$ ways. From the remaining $n-j$ \red elements and $k$ \blue elements, a Callan sequence of size $(n-j) \times k$ can be constructed in $B_{n-j}^{(-k)}(0)$ ways. Hence, we have
	\[
		\sum_{C \in \mathcal{AC}(n,k)} x^{w_{\mathcal{AC}}(C)} = \sum_{j=0}^n \binom{n}{j} B_{n-j}^{(-k)} (0) x^j,
	\]
	which equals $B_n^{(-k)}(x)$.
\end{proof}

We provide a similar claim to \cref{theo:pBpol_closed form}.

\begin{theorem}
	For $n,k \geq 0$, we have
	\[
	B_n^{(-k)}(x)=\sum_{m=0}^{\min(n,k)}(m!)^2\sts{k+1}{m+1}\sts{n}{m}_{x+1}.
	\]
\end{theorem}

\begin{proof}
	Let $N = \{\p{1}, \ldots, \p{n}\} \cup \{\p{*}\} \cup \{\p{\dagger}\}$  (we added $\p{\dagger}$ to mark the special block $S$) and $K = \{\k{1}, \ldots, \k{k}\} \cup \{\k{*}\}$. 
	Take an ordered partition of $K$ into $m+1$ non-empty blocks $B_1, \dots, B_m, B^*$. This can be done in $m! \sts{k+1}{m+1}$ ways. 
	We consider two cases separately for partitions of $N$. 
	
	First, we partition the set $N \setminus\{\p{*}\}$ into $m+1$ non-empty blocks. 
	By pairing the blocks that do not contain $\p{\dagger}$ with the $m$ ordinary blocks of \blue elements, and by deleting $\p{\dagger}$, we obtain an abundant Callan sequence such that $R^*$ contains only $\p{*}$. 
	Then the number of such partitions of $N$ with weight is $m! \sts{n}{m}_x$.
	
	Second, we partition the set $N\setminus\{\p{*}\}$ into $m+2$ non-empty blocks. 
	The block containing $\p{\dagger}$ becomes the special block $S$ after deleting $\p{\dagger}$. 
	We pair up the remaining $m+1$ blocks of \red elements with the $m+1$ blocks of \blue elements to obtain an abundant Callan sequence with $|R^*| > 1$. The number of such partitions of $N$ is $(m+1)! \sts{n}{m+1}_x$.

    Hence, 
	\[
	\sum_{C \in \mathcal{AC}(n,k)} x^{w_{\mathcal{AC}}(C)} = \sum_{m=0}^{\min(n,k)} m! \sts{k+1}{m+1} \left((m+1)! \sts{n}{m+1}_x + m! \sts{n}{m}_x \right).
	\]
	The theorem follows now by \cref{rec-stirling-sec}. 
\end{proof}

We define Stirling polynomials of the first kind, $\stf{n}{m}_x$, for arbitrary $x$ in a similar manner as Stirling polynomials of the second kind. 
It is well known that the Stirling number of the first kind $\stf{n}{m}$ is the number of permutations of $n$ elements with $m$ cycles. 
It is also known that the permutations with $m$ cycles are in bijections with permutations with $m$ left-to-right minima. 
(For a permutation $p=p_1p_2\cdots p_n$, the element $p_i$ is called a \emph{left-to-right minimum} of $p$ if $p_i<p_j$ for all $j<i$).

For non-negative integers $n$ and $m$, let $\mathcal{S}_1(n,m)$  denote the set of permutations of $\{0,1,2,\ldots, n\}$ with $m+1$ cycles. 
We record the cycles starting with the smallest element in the cycle, and the cycles are listed according to their smallest element in increasing order.

We define the weight function $w_{\mathcal{S}_1}: \mathcal{S}_1(n,m) \to \Z_{\geq 0}$ by the number of the left-to-right minima in the sequence defined by removing $0$ from the first cycle. 
 
\begin{example}\label{Example: Stirling-poly}
	We identify a permutation $p = p_0 p_1 p_2 \cdots p_n$ of $\{0, 1, 2, \dots, n\}$ with a bijection $p: \{0,1,\dots,n\} \to \{0,1, \dots, n\} ; i \mapsto p_i$ as usual. 
	For example, a permutation $p = 4960537281$ of $\{0, 1, \dots, 9\}$ is identified with the product of $4$ cycles $(0, 4, 5, 3) (1, 9) (2, 6, 7) (8)$. 
	The order of cycles is determined by their smallest elements, namely, $0, 1, 2$, and $8$. The sequence defined by removing $0$ from the first cycle is $453$. 
	There are $2$ left-to-right minima in $453$, namely, $4$ and $3$. Thus we have $w_{\mathcal{S}_2}(p) = 2$.
\end{example}

\begin{definition}\label{Stirling-poly-1st-def} 
For non-negative integers $n$ and $m$ with $m\leq n$, we define the \emph{Stirling polynomial of the first kind} as
	\[
	\stf{n}{m}_x = \sum_{p \in \mathcal{S}_1(n,m)} x^{w_{\mathcal{S}_1}(p)}.
	\]
\end{definition}

\begin{example}
	There are $11$ elements in $\mathcal{S}_1(3,1)$, namely,
	\begin{align*}
		&(0) (1,2,3) & &(0) (1,3,2)\\
		&(0,\underline{1})(2,3) & &(0,\underline{2})(1,3) & &(0,\underline{3})(1,2) & &(0,\underline{2},3)(1) & &(0,\underline{1},3)(2) & &(0,\underline{1},2)(3)\\
		&(0,\underline{3},\underline{2})(1) & &(0,\underline{3},\underline{1})(2) & &(0,\underline{2},\underline{1})(3).
	\end{align*}
	Thus, we have $\stf{3}{1}_x = 3x^2 + 6x + 2$.
\end{example}

\begin{proposition}\label{prop:stirling first}
	For $n \geq 0$, it holds
	\[
	(x+y)_n := (x+y)(x+y+1) \cdots (x+y+n-1) = \sum_{m=0}^n \stf{n}{m}_y x^m.
	\]
	The Stirling polynomials of the first kind satisfy the recursion
	\[
	\stf{n+1}{m}_x = \stf{n}{m-1}_x + (x+n) \stf{n}{m}_x
	\]
	with initial values $\stf{0}{0}_x = 1, \stf{n}{0}_x = (x)_n = x(x+1) \cdots (x+n-1)$, and $\stf{0}{m}_x = 0$.
\end{proposition}

\begin{proof}
	We define a map $f: \mathcal{S}_1(n+1,m) \to \mathcal{S}_1(n,m)$ by deleting $n+1$. For each $p \in \mathcal{S}_1(n+1,m)$, the following three cases occur.
	\begin{enumerate}
		\item If $p_0 = n+1$, that is, the first cycle has of the form $(0, n+1, \dots )$, then the number of cycles is preserved and the weight is decreased by one.
		\item If $p_{n+1} = n+1$, that is, $(n+1)$ is a cycle of $p$, then the number of cycles is decreased by one and the weight is preserved.
		\item If otherwise, then the number of cycles and the weight are preserved.
	\end{enumerate}
	Therefore, we have
	\[
		\stf{n+1}{m}_x = x \stf{n}{m}_x + \stf{n}{m-1}_x + n \stf{n}{m}_x.
	\]
	The recursion is translated into the generating function
	\[
		(x+y)(x+y+1) \cdots (x+y+n-1) = \sum_{m=0}^n \stf{n}{m}_y x^m,
	\]
	which concludes the proof.
\end{proof}

The motivation of introducing the Stirling polynomials of the first kind as above was to provide a combinatorial proof for the following theorem.

\begin{theorem}\label{theorem:Ber-hypersum}
	For $n,k \geq 0$, it holds
	\begin{align}\label{form: cycles_pB}
		\sum_{j=0}^n \stf{n}{j}_y B_j^{(-k)}(x) = \sum_{\ell=0}^n {n \choose \ell} (x+y)_{n-\ell} \ell! (\ell+1)^k.
	\end{align}
\end{theorem}

Before proving the theorem, we define a special combinatorial object. 
\begin{definition}
Let $\mathcal{CS}(n,k)$ denote the set of permutations of the set $\{0, 1, \dots, n\} \cup \{\k{1}, \dots, \k{k}\} \cup \{|\}$ satisfying the following two conditions. 
\begin{itemize}
	\item All substrings of \blue elements are in increasing order.
	\item All \blue elements are located before the bar, and $0$ is after the bar.
\end{itemize}
\end{definition}

We construct a bijection $F: \bigcup_{j=0}^n (\mathcal{AC}(j,k) \times \mathcal{S}_1(n,j)) \to \mathcal{CS}(n,k)$ by following the steps below.

\begin{enumerate}
	\item For an abundant Callan sequence $(B_1; R_1) \cdots (B_m; R_m) \cup (B, \k{*}; R, \p{*}) \cup S \in \mathcal{AC}(j,k)$, we deform it to $R B_1 R_1 \cdots B_m R_m B \p{*} S$.
	\item Replace $\p{i}$ in the above sequence by the $i$-th cycle $c_i$ of the permutation $p = c_0 c_1 \cdots c_j \in \mathcal{S}_1(n,j)$, and replace $\p{*}$ by $``| c_0"$ the bar and the $0$-th cycle containing $0$.
	\item Identify every product of cycles with a permutation of their entries. 
\end{enumerate}

\begin{example}	
	As an example, take an abundant Callan sequence 
	\[
		C = (\k{1}, \k{3}, \k{4}; \p{1}) (\k{6}, \k{7}; \p{2}, \p{3}) \cup (\k{2}, \k{5}, \k{*}; \p{5}, \p{*}) \cup (\p{4}) \in \mathcal{AC}(5,7)
	\]
	and a permutation (product of cycles) $p = 8371546209 = (0,8)(1,3)(2,7)(4,5)(6)(9) \in \mathcal{S}_1(9,5)$. We obtain a permutation in $\mathcal{CS}(9,7)$ by the above procedure 
	\begin{align*}
		C &\leftrightarrow \p{5} \k{1}\k{3}\k{4} \p{1} \k{6}\k{7} \p{2}\p{3} \k{2}\k{5} \p{*} \p{4}\\
		&\leftrightarrow (9) \k{1}\k{3}\k{4} (1,3) \k{6}\k{7} (2,7) (4,5) \k{2}\k{5} \mid (0,8) (6)\\
		&\leftrightarrow 9 \k{1}\k{3}\k{4} 31 \k{6}\k{7} 7542 \k{2}\k{5} \mid 860.
	\end{align*}
\end{example}

\begin{proof}[Proof of \cref{theorem:Ber-hypersum}]
We define two weight functions $w_{\mathcal{CS},1}, w_{\mathcal{CS},2}: \mathcal{CS}(n,k) \to \mathbb{Z}_{\geq 0}$ by $w_{\mathcal{CS},1}(F(C,p)) = w_\mathcal{AC}(C)$ and $w_{\mathcal{CS},2}(F(C,p)) = w_{\mathcal{S}_1}(p)$. 
We now consider the polynomial with two variables defined by
\[
	\mathcal{CS}_n^k(x,y) = \sum_{s \in \mathcal{CS}(n,k)} x^{w_{\mathcal{CS},1}(s)} y^{w_{\mathcal{CS},2}(s)}.
\]
By the definition of the weights, we have
\begin{align*}
	\mathcal{CS}_n^k(x,y) &= \sum_{j=0}^n \sum_{C \in \mathcal{AC}(j,k)} x^{w_\mathcal{AC}(C)} \sum_{p \in \mathcal{S}_1(n,j)} y^{w_{\mathcal{S}_1}(p)}\\
	&= \sum_{j=0}^n B_j^{(-k)}(x) \times \stf{n}{j}_y.
\end{align*}

On the other hand, elements $s \in \mathcal{CS}(n,k)$ can be counted as follows.

First, we divide the set $\{0, 1, \dots, n\}$ into two sets. One of the sets contains $\ell$ elements without the special element $0$, so we have $\binom{n}{\ell}$ options. 
Then, we permute the $\ell$ chosen elements in $\ell !$ ways and insert $k$ \blue elements between (before and after) these $\ell$ \red elements in $(\ell + 1)^k$ ways. 
The \blue elements in each substring are ordered increasingly.

Second, the permutation of the remaining $n-\ell+1$ elements including $0$ form $m+1$ cycles ($0 \leq m \leq n-\ell$). 
Recall that the number of cycles, $m$, (without the cycle containing $0$) equals $w_{\mathcal{CS},1}(s)$. Thus,
\begin{align*}
	\sum_{s \in \mathcal{CS}(n,k)} x^{w_{\mathcal{CS},1}(s)} y^{w_{\mathcal{CS},2}(s)} &= \sum_{\ell=0}^n {n \choose \ell} \ell! (\ell+1)^k \sum_{m=0}^{n-\ell} x^m \sum_{p \in \mathcal{S}_1(n-\ell, m)} y^{w_{\mathcal{S}_1}(p)}\\
	&= \sum_{\ell=0}^n {n \choose \ell} \ell! (\ell+1)^k \sum_{m=0}^{n-\ell} x^m \stf{n-\ell}{m}_y.
\end{align*}
Since the inner sum equals $(x+y)_{n-\ell}$ by \cref{prop:stirling first}, the assertion holds.
\end{proof}

Starting from the famous sum of powers of integers
\begin{align*}
	S_k^{(0)}(n) = S_k(n) = 1^k+2^k+\cdots +n^k,
\end{align*}
Faulhaber (see Knuth~\cite{Knuth1993}) introduced recursively \emph{hypersums} as

\begin{align*}
	S_k^{(r)}(n) = \sum_{j=1}^{n} S_k^{(r-1)}(j).
\end{align*}

Karg$\i$n--Cenkci--Dil--Can~\cite{KarginCenkciDilCan2020} presented an identity connecting poly-Bernoulli polynomials and hypersums \cite[Theorem 1]{KarginCenkciDilCan2020}. 
\cref{theorem:Ber-hypersum} implies this result. 

\begin{corollary}\cite[Theorem 1]{KarginCenkciDilCan2020}
For non-negative integers $n$, $k$, $m$, and $r$ with $m+r > 0$, we have
\[
\sum_{j=0}^n \stf{n}{j}_r B_j^{(-k)}(m) = n! S_k^{(m+r-1)}(n+1).
\]	
\end{corollary}

\begin{proof}
	For $x = m, y = r$, the right-hand side of \eqref{form: cycles_pB} becomes
	\[
	n! \sum_{\ell=1}^{n+1} {m+r+n-\ell \choose m+r-1} \ell^k.
	\]
	By~\cite[Corollary 2]{LaissaouiBounebiratRahmani2017}, this equals $n! S_k^{(m+r-1)}(n+1)$.
\end{proof}

\begin{corollary} \cite[Corollary 24]{BenyiMatsusaka2021}
For $m=0$ and $r=1$, we have
\[
\sum_{j=0}^n \stf{n+1}{j+1} B_j^{(-k)} = n! S_k(n+1).
\]
\end{corollary}

\section{Poly-Euler numbers}\label{section:poly-Euler} \label{s3}

Euler numbers count the alternating permutations of a set with an even number of elements. 
They arise in the Taylor series expansion of the secant and hyperbolic secant function,
\begin{align*}
	 \sum_{n=0}^{\infty}E_n\frac{t^n}{n!}=\frac{1}{\cosh t}.
\end{align*}
The odd indexed Euler numbers are zero, the even indexed numbers are 
$1$, $-1$, $5$, $-61$, $1385$, $\ldots $  A028296 in OEIS~\cite{OEIS}.
Ohno--Sasaki~\cite{OhnoSasaki2012, OhnoSasaki2017} introduced poly-Euler numbers 
\begin{align*}
	\sum_{n=0}^{\infty} E^{(k)}_n \frac{t^n}{n!} = \frac{\li_k(1-e^{-4t})}{4t \cosh t}
\end{align*}
using polylogarithm function in a similar manner as poly-Bernoulli numbers. Due to Sasaki~\cite{Sasaki2012}, poly-Euler numbers are introduced as special values of a certain $L$-function. 
The number $E^{(1)}_n$ is nothing but the original Euler number $E_n$. 

Komatsu and Zhu \cite{KomatsuZhu2020} introduced \emph{complementary Euler numbers}, $\widehat{E}_n$, as a special case of \emph{hypergeometric Euler numbers} (see also \cite{Komatsu2017}) by the generating function
\begin{align*}
	\sum_{n=0}^{\infty} \widehat{E}_n\frac{t^n}{n!} = \frac{t}{\sinh t}.
\end{align*}
Note that Koumandos and Laurberg Pedersen~\cite{KoumandosLaurbergPedersen2012} called this sequence as \emph{weighted Bernoulli numbers}. 

As a generalization of the complementary Euler numbers, Komatsu~\cite{Komatsu2017, Komatsu2020} defined and studied \emph{poly-Euler numbers of the second kind}, $\widehat{E}_n^{(k)}$ by the generating function
\begin{align*}
	\sum_{n=0}^{\infty} \widehat{E}_n^{(k)} \frac{t^n}{n!} =\frac{\li_k(1-e^{-4t})}{4\sinh t}.
\end{align*}

The arrays $\widetilde{E}_n^{(-k)}=nE^{(-k)}_{n-1}$ and $\widehat{E}_n^{(-k)}$ are non-negative integers. \cref{table:poly_Euler} shows the first few values of these arrays.

 \begin{table}[H]
 	\begin{tabular}{|c|cccc|}
 		\hline
 		$n,k$ &0&1&2&3\\\hline
 		0&0&0&0&0\\
 		1& 1&1&1&1\\
 		2&4&12&28&60\\
 		3&13&109&493&1837\\
 		4&40&888&7192& 42840\\\hline
 		
 	\end{tabular}\hspace{1cm}
 	\begin{tabular}{|c|cccc|}
 	\hline
 	$n,k$ &0&1&2&3\\\hline 
 	0& 1&1&1&1\\
 	1&2&6&14&30\\
 	2&5&37&165&613\\
 	3&14&234&1826& 10770  \\
 	4 & 41 & 1513& 19689& 175465 \\\hline
 \end{tabular}
 	\caption{The poly-Euler numbers $\widetilde{E}_n^{(-k)}$ and $\widehat{E}_n^{(-k)}$.}
 	\label{table:poly_Euler}
 \end{table}

In this section, we extend the notion of abundant Callan sequences and define \emph{$E$-sequences}. 
As we will see $E$-sequences are appropriate for a combinatorial study of poly-Euler numbers. 
 
\begin{definition}
An $(n,k)$-\emph{$E$-sequence} is a sequence that we obtain from an abundant Callan sequence of size $n\times k$ by associating to each \red element in ordinary blocks and in the extra block a sign, $+$ or $-$, and separating the ordinary blocks of \red elements into two parts, left and right part. 
Further, we call an $(n,k)$-$E$-sequence \emph{odd} if the size of $S$ is odd, \emph{even} if $|S|$ is even.
\end{definition}

\begin{example}\label{example_first}
An odd $(12,9)$-$E$-sequence is built from the elements $\{\p{1},\ldots, \p{12}\}$ and $\{\k{1},\ldots, \k{9}\}$. Let  $S=\{\p{2},\p{5},\p{12}\}$ and the sequence of ordinary pairs and the extra pair is 
	\begin{align*}
		&(\k{2}, \k{4}, \k{5}; \emptyset \mid \p{+4}, \p{+11})(\k{1}, \k{7}, \k{8};\p{-7} \mid \emptyset)(\k{3}; \p{-1} \mid \p{-6}, \p{+8})(\k{6}, \k{9}; \p{+3} \mid \p{-9})\\
		&\cup (\k{*}; \p{-10}, \p{*}).
	\end{align*}
In the example, we indicated the left and right parts of the ordinary blocks of \red elements by separating them with a bar. 
Note that it is allowed that one (but only one) of these parts is empty. 
\end{example}

\begin{example}
	An odd $(1,k)$-$E$-sequence is a sequence on the elements $\{\p{1}\}\cup\{\k{1},\ldots, \k{k}\}$. The unique \red element has to be contained in $S$, so the number of such sequences is always $1$. This explains combinatorially Corollary 6.3 in~\cite{OhnoSasaki2017}.
\end{example}

\begin{example}
	We count the cases of odd $(3,2)$-$E$-sequences. 
	\begin{itemize}
		\item[Case 1.] If $|S|=3$, then we have $1$ odd $E$-sequence.
		\item[Case 2.] If $|S|=1$, then two \red elements remain, $\{\p{1},\p{2}\}$, $\{\p{1},\p{3}\}$, or $\{\p{2},\p{3}\}$. The cases with  $\{\p{1},\p{2}\}$ are 
		\begin{itemize}
			\item If $\k{1}, \k{2}$ are in the same block, then we have $9$ cases:
		\begin{align*}
			&(\k{1}, \k{2}, \k{*}; \p{1}, \p{2}, \p{*}) && (\k{1}, \k{2}; \p{1}, \p{2} \mid \emptyset)(\k{*}; \p{*}) &&(\k{1}, \k{2}; \emptyset \mid \p{1}, \p{2})(\k{*};\p{*})\\
			&(\k{1},\k{2};\p{2} \mid \emptyset)(\k{*}; \p{1}, \p{*}) && (\k{1}, \k{2}; \emptyset \mid \p{2})(\k{*}; \p{1}, \p{*}) && (\k{1}, \k{2}; \p{2} \mid \p{1})(\k{*}; \p{*})\\
			&(\k{1}, \k{2}; \p{1} \mid \emptyset)(\k{*}; \p{2}, \p{*}) && (\k{1}, \k{2}; \emptyset \mid \p{1})(\k{*}; \p{2}, \p{*}) && (\k{1}, \k{2}; \p{1} \mid \p{2})(\k{*};\p{*})
		\end{align*}
			\item  If $\k{1}, \k{2}$ are not in the same block but $\p{1}, \p{2}$ are in the same block, then we have $4$ cases:
		\begin{align*}
			&(\k{1}; \p{1}, \p{2} \mid \emptyset)(\k{2}, \k{*}; \p{*}) && (\k{1}; \emptyset \mid \p{1}, \p{2})(\k{2}, \k{*}; \p{*})\\
			&(\k{2}; \p{1}, \p{2} \mid \emptyset)(\k{1}, \k{*}; \p{*}) && (\k{2}; \emptyset \mid \p{1}, \p{2})(\k{1}, \k{*}; \p{*}) 
		\end{align*}
			
			\item If $\k{1}$ and $\k{2}$ are in distinct blocks, and  $\p{1}$ and $\p{2}$ are in distinct blocks, then we have $28$ cases. 
		\begin{align*}
			&(\k{1}; \p{1} \mid \emptyset)(\k{2}, \k{*}; \p{2}, \p{*}) && (\k{1}; \emptyset \mid \p{1})(\k{2}, \k{*}; \p{2}, \p{*}) && (\k{1}; \p{1} \mid \p{2})(\k{2}, \k{*}; \p{*})\\
			&(\k{1}; \p{1}\mid \emptyset)(\k{2};\p{2} \mid \emptyset)(\k{*};\p{*}) && (\k{1}; \p{1} \mid \emptyset)(\k{2}; \emptyset \mid \p{2})(\k{*};\p{*})\\
			&(\k{1}; \emptyset \mid \p{1})(\k{2}; \p{2} \mid \emptyset)(\k{*};\p{*}) && (\k{1}; \emptyset \mid \p{1})(\k{1}; \emptyset \mid \p{2})(\k{*}; \p{*})
		\end{align*}
			In the list above, we can exchange the places of $\p{1}$ and $\p{2}$, similarly $\k{1}$ and $\k{2}$, and obtain so the 28 cases.
		\end{itemize}
		Now if we assign $+$ or $-$ to each \red element we have to multiply the number of cases by $4$. 
		
		So, $(9+4+28)\cdot 4=164$. In all the 3 cases ($\{\p{1},\p{2}\}$ $\{\p{1},\p{3}\}$ or $\{\p{2},\p{3}\}$) there are 164 cases, so we have all together $3\cdot 164=492$ possibilities. 
	\end{itemize}
	Hence, we have $493$ odd $(3,2)$-$E$-sequences.
\end{example}

\begin{theorem}
	The poly-Euler number $\widetilde{E}_{n}^{(-k)}$ counts odd $(n,k)$-$E$-sequences. 
\end{theorem}

\begin{proof}
To prove the theorem, we recall an inclusion-exclusion type formula shown by Ohno and Sasaki~\cite[Corollary 6.6]{OhnoSasaki2017},

\begin{align}\label{form:poly_euler_inc_excl}
	\widetilde{E}_n^{(-k)} = \sum_{\eal{m=1}{m:\text{odd}}}^{n} \binom{n}{m} \sum_{\ell=0}^k (-1)^{k-\ell} \ell!\sts{k}{\ell}2^{n-m}(2\ell+1)^{n-m}
\end{align}
for any non-negative integers $n$ and $k$.

We show that odd $E$-sequences are enumerated by \eqref{form:poly_euler_inc_excl}. 
Let $|S|=m$, i.e., we choose from the $n$ \red elements $m$ elements that will not be used in forming the Callan pairs ($\binom{n}{m}$ possibilities). 

The remaining $n-m$ \red elements and $k$ \blue elements form a sequence of pairs. To count them, we similarly use the inclusion-exclusion principle to the proof of \cref{Bnkr-anti}. Our objects here are permutations of these elements,
\[
	R_0 B_1 R_1 B_2 R_2 \cdots B_j R_j \quad (1 \leq j \leq k),
\]
satisfying the following conditions.
\begin{enumerate}
	\item In each block, all \red elements are in increasing order. Note that the blocks $R_0$ and $R_j$ can be empty.
	\item Each \red element in $R_1, \dots, R_j$ has two data $(p, s) \in \{\mathrm{left}, \mathrm{right}\} \times \{+, -\}$. Each \red element in $R_0$ has one data $s \in \{+,-\}$. 
\end{enumerate}
If all \blue elements in each $B_i$ are also in increasing order, then we can similarly obtain an $E$-sequence as in the proof of \cref{Bnkr-anti}. For instance, the sequence in \cref{example_first} is given from
\begin{align*}
	\p{10}^-, \k{2}, \k{4}, \k{5}, \p{4}^{r,+}, \p{11}^{r,+}, \k{1}, \k{7}, \k{8}, \p{7}^{l,-}, \k{3}, \p{1}^{l,-}, \p{6}^{r,-}, \p{8}^{r,+}, \k{6}, \k{9}, \p{3}^{l,+}, \p{9}^{r,-}.
\end{align*}

For a subset $A \subset \{\k{2}, \dots, \k{k}\}$, let $\alpha(A)$ (resp. $\beta(A)$) denote the number of permutations $R_0B_1R_1 \cdots B_j R_j$ with $D(B_1) \cup \cdots D(B_j) = A$ (resp. $D(B_1) \cup \cdots D(B_j) \supset A$), where $D(B_i)$ is the descent set of $B_i$. Then by the above observation, we see that
\[
	\widetilde{E}_n^{(-k)} = \sum_{\substack{m=1 \\ m: \text{odd}}}^n {n \choose m} \alpha(\emptyset).
\]
For $k \geq 2$, by the inclusion-exclusion principle,
\[
	\alpha(\emptyset) = \sum_{\ell=0}^{k-1} (-1)^\ell \sum_{\substack{A \subset \{\k{2}, \dots, \k{k}\}\\|A| = \ell}} \beta (A).
\]
First, we construct $(k-\ell)$-non-empty ordered blocks of $\{\k{1}, \dots, \k{k}\}$ in $\sts{k}{k-\ell} (k-\ell)!$ ways. 
In each block, we record the elements in decreasing order. For each \red element, there are $2(k-\ell)+1$ possibilities for choosing the position and the data $p \in \{\mathrm{left}, \mathrm{right}\}$, and $2$ possibilities to choose the sign $s \in \{+,-\}$. 
If two blocks of \blue elements are consecutive, then we connect them. The resulting permutations have at least $\ell$ descents in blocks of \blue elements. 
Therefore we have
\[
	\sum_{\substack{A \subset \{\k{2}, \dots, \k{k}\} \\ |A| = \ell}} \beta(A) = \sts{k}{k-\ell} (k-\ell)! (2(k-\ell)+1)^{n-m} 2^{n-m},
\]
that is,
\[
	\alpha(\emptyset) = \sum_{\ell=0}^k (-1)^{k-\ell} \sts{k}{\ell} \ell! (2\ell+1)^{n-m} 2^{n-m}.
\]
For the cases in $k = 0$ and $k=1$, we can check the above formula for $\alpha(\emptyset)$ also works.
\end{proof}

\begin{theorem}
	The poly-Euler number of the second kind $\widehat{E}_{n}^{(-k)}$ counts even $(n,k)$-$E$-sequences. 
\end{theorem}
\begin{proof}
	Komatsu~\cite[Corollary 3.5]{Komatsu2017} showed the analogue formula of \eqref{form:poly_euler_inc_excl} for the poly-Euler number of the second kind,
	\begin{align*}
		\widehat{E}_n^{(-k)} = \sum_{\eal{m=0}{m:\text{even}}}^{n}\binom{n}{m}\sum_{\ell=0}^{k} (-1)^{k-\ell} \ell!\sts{k}{\ell}2^{n-m} (2\ell+1)^{n-m}.
	\end{align*}
By the previous argument, the theorem follows. 
\end{proof}

We explicitly consider the interesting special case when $k=0$.
The special case $\widetilde{E}_n^{(0)}$ is the number sequence A003462 in~\cite{OEIS}, while $\widehat{E}_n^{(0)}$ is the sequence A007051 in~\cite{OEIS}. 
Due to our interpretation, these sequences enumerate odd and even $E$-sequences containing only \red elements. If there are no \blue elements, there are no ordinary pairs and the $E$-sequence is constructed with the sets $R^*$ and $S$ of \red elements. 
However, the elements in $R^*$ can be $+$ or $-$. This means that the number of such $E$-sequences is the number of ways of selecting the $n$ elements into three sets: $S$, $R^*(+)$ and $R^*(-)$, 
where $R^*(+)$ denotes the set of \red elements in the extra block that were assigned by $+$, ($R^*(-)$ is defined similarly).  
Hence, we have

\begin{align*}
	\widetilde{E}_{n}^{(0)} = \sum_{j=0}^{\lfloor \frac{n-1}{2} \rfloor}\binom{n}{2j+1} 2^{n-2j-1}\quad \mbox{and}\quad \widehat{E}_{n}^{(0)} = \sum_{j=0}^{\lfloor \frac{n}{2} \rfloor}\binom{n}{2j} 2^{n-2j}.
\end{align*} 
On the other hand, we have another closed formula. 

\begin{lemma}\label{lemma:pE_special_case}
	For $n \geq 0$, it holds
	\begin{align*}
		\widetilde{E}_n^{(0)} = \frac{3^{n}-1}{2}\quad \mbox{and}\quad \widehat{E}_n^{(0)} = \frac{3^{n}+1}{2}.
	\end{align*}
\end{lemma}

\begin{proof}
	Let $0$ mark the elements that are in $S$. Then an $E$-sequence with only \red elements is determined by assigning to each element a ``sign" $\{0,+,-\}$.  We define an involution on the set of words, $w$, on the alphabet $\{0,+,-\}$ of length $n$. 
	The involution is based on the first entry not equal to $-$, i.e., $\min i$ with $w_i\not = -$. Exchange the $i$-th entry: if $w_i=+$ change it to $w_i = 0$ and if it is $w_i=0$ change it to $w_i=+$. 
	It is clearly a bijection between the set of words with an odd number of zeros and an even number of zeros. 
	The involution can be applied to all words except the word with $w_j=-$ for all $j$. This word has an even number of zeros. 
	We proved that the number of words with an even number of zeros is $(3^{n}+1)/2$, while the number of words with an odd number of zeros is $(3^{n}-1)/2$.
\end{proof}

Ohno and Sasaki~\cite[Corollary 6.3]{OhnoSasaki2017} mention another interesting special case, $n=2$,
\begin{align*}
	\widetilde{E}_2^{(-k)}=4(2^{k+1}-1).
\end{align*}
	We can see this equality in an elementary way based on our interpretation as follows. If there are two \red elements, exactly one of them creates the set $S$. The other \red element has a sign $+$ or $-$. 
	If this \red element is in the extra block, then all the \blue elements have to be in the extra block, so there is one such odd $E$-sequence. If this \red element is in the single ordinary block, then the \blue elements go to the ordinary or to the extra block, which gives $2^{k}-1$ possibilities. 
	(Each element has a choice to which block it goes ($2^{k}$ possibilities), but the case when all would go to the extra block is not allowed.) 
	Finally, we have to determine the part (left or right) of the ordinary block for the \red element. So, the number of $(2,k)$-$E$-sequences is $2^2(2(2^{k}-1)+1)$.   

Next, we present combinatorial proofs for some identities involving the poly-Euler numbers. 
Analogous results can be derived for the poly-Euler numbers of the second kind, but we omit the explicit proofs here. 

\begin{theorem}\label{theo:poly_euler_closed}
We have the following combinatorial formula for $\widetilde{E}_n^{(-k)}$,
		\begin{align*}
			\widetilde{E}_n^{(-k)}=	\sum_{j=0}^{\lfloor\frac{n-1}{2}\rfloor} \binom{n}{2j+1}2^{n-2j-1}\sum_{m= 0}^km!\sts{k+1}{m+1}\sum_{s=0}^{m} 2^{m-s}\binom{m}{s}(m+s)!\sts{n-2j}{m+s+1}.
		\end{align*}
\end{theorem}

\begin{proof}
	Choose $2j+1$ \red elements for $S$ in $\binom{n}{2j+1}$ ways and construct from the remaining elements a sequence of pairs of blocks as follows. 
	Let $m$ be the number of ordinary pairs.
	 
	First, construct an ordered partition of the \blue elements, $\{\k{1},\ldots, \k{k},\k{*}\}$, with $m+1$ blocks such that the block containing the element $\k{*}$ is the last. This can be done in $m!\sts{k+1}{m+1}$ ways.
	 
	Now we turn our attention to the \red elements. We assign to each \red element a sign in $2^{n-2j-1}$ ways.
	We have to create $m$ ordinary blocks with two parts and an extra block from the remaining $n-2j-1$ \red elements.
	
	Let $s$ be the number of those ordinary blocks where none of the parts is empty. In other words in $m-s$ ordinary blocks, one of the parts (left or right) is empty. 
	This means that we need to create $2s+(m-s)+1$ blocks from the \red elements. 
	This can be done in $\sts{n-2j}{m+s+1}$ ways. (The $+1$ stands for the extra element $\p{*}$). The block containing $\p{*}$ is paired with the extra block of \blue elements. 
	
	Take a linear arrangement of the $m+s$ ordinary blocks, $\pi$. Let the first $m$ blocks in $\pi$ be paired to the ordinary blocks of \blue elements in the same order as in $\pi$. 
	Now choose $s$ from the $m$ ordinary pairs in $\binom{m}{s}$ ways and add the \red elements in the $(m+i)$-th block as the right part in the $i$-th chosen pair. 
	In the case of the other ordinary pairs, so where we added only one block of \red elements, choose to which part the block should belong in $2^{m-s}$ ways.
\end{proof}

We give combinatorial proof for a formula given by Ohno and Sasaki~\cite[Theorem 10.1]{OhnoSasaki2017}.
\begin{theorem}\cite{OhnoSasaki2017}\label{theorem:Ohno-Sasaki}
	For $n > 0$ and $k \geq 0$, we have 
	\begin{align*}
		\widetilde{E}_{n}^{(-k)} = \sum_{m=0}^{\min(n-1,k)}(m!)^2 \left(\sum_{\ell=1}^{n-m} \binom{n}{\ell}\widetilde{E}^{(0)}_{\ell}\sts{n-\ell}{m} 4^{n-\ell}\right)\sts{k+1}{m +1}.
	\end{align*}
\end{theorem}

\begin{proof}
	Let $\ell$ be the number of \red elements in $S\cup R^*$. We know that the number of creating the sets is $\binom{n}{\ell}\widetilde{E}_{\ell}^{(0)}$. 
	Let $m$ be the number of ordinary blocks in the $E$-sequence, so $\ell$ goes from $1$ ($S$ is non-empty) to $n-m$ because every ordinary pair has to contain at least one \red element.
	The remaining $n-\ell$ \red elements are partitioned into $m$ ordered blocks in $m!\sts{n-\ell}{m}$ ways. 
	Further, to each element we add a sign and a position (left or right part) in $4^{n-\ell}$ ways. 
	Finally, the ordered partition of \blue elements can be constructed in $m!\sts{k+1}{m+1}$ ways. 
\end{proof}

Our next goal is to show combinatorially some relations between the poly-Euler numbers and poly-Bernoulli numbers. 
First, we present the connection with the poly-Bernoulli numbers of type $C$ introduced in~\cite{ArakawaKaneko1999_2}. 
In our setting, these numbers are the values of the poly-Bernoulli polynomials at $x=-1$, that is, $C_n^{(k)} = B_n^{(k)}(-1)$. 
Similarly, as the poly-Bernoulli numbers $B_n^{(k)} = B_n^{(k)}(0)$ (sometimes called poly-Bernoulli numbers of type $B$), the numbers with negative $k$ indices are integers, and have several combinatorial interpretations (see~\cite{BenyiHajnal2017}). 
Most of the objects that are enumerated by $C_n^{(-k)}$ are slight modifications of those enumerated by $B_n^{(-k)}$. However, some objects arise naturally in the context of type $C$. 
Since the following fact is important for our next result, we recall it as a lemma. 

\begin{lemma}\label{lemma: pB_type_C}\cite{BenyiHajnal2017}
	The number of $(n,k)$-Callan sequences such that the extra \red block contains only $\p{*}$ is $C_n^{(-k)}$. 	
\end{lemma}

\begin{theorem}\label{theorem: pE_pBC}
	For $n, k \geq 0$, it holds
	\begin{align}\label{formula:pE_pBC_1} 
		\widetilde{E}_{n}^{(-k)} = \sum_{\ell=0}^{n} \binom{n}{\ell} \widetilde{E}_{\ell}^{(0)} C_{n-\ell}^{(-k)} 4^{n-\ell}\quad \mbox{and}\quad		\widehat{E}_{n}^{(-k)} = \sum_{\ell=0}^{n} \binom{n}{\ell} \widehat{E}_{\ell}^{(0)} C_{n-\ell}^{(-k)} 4^{n-\ell}.
	\end{align}
\end{theorem}

\begin{proof}
	The proof is based on the same idea as that of \cref{theorem:Ohno-Sasaki}. Let $\ell$ be the number of \red elements in the set $S\cup R^{*}$.
	Take now an $(n-\ell, k)$-Callan sequence with the property that the extra block $R^*$ contains only the element $\p{*}$. 
	Add a sign and a position (left or right part) to each \red element in the Callan sequence. We can combine now our ``decorated" Callan sequence with the sets $S$ and $R^*$ to obtain a valid $E$-sequence. 
	We choose the $\ell$ elements in $\binom{n}{\ell}$ ways, we construct the sets  $S$ and $R^*$ in $\widetilde{E}_{\ell}^{(0)} $ ways and the decorated Callan sequences in $C_{n-\ell}^{(-k)} 4^{n-\ell}$ ways. 
\end{proof}

We recall here the special case $k=1$, that is, the Corollary 10.2 in~\cite{OhnoSasaki2017} and its analogue for the poly-Euler numbers of the second kind.
\begin{corollary}\cite{OhnoSasaki2017}\label{corollary:pE_special_1}
	For $n \geq 0$, we have 
	\begin{align*}
		\widetilde{E}_{n}^{(-1)} = \sum_{\ell=0}^{n}\binom{n}{\ell} \widetilde{E}_{\ell}^{(0)}4^{n-\ell} \quad \mbox{and}\quad \widehat{E}_{n}^{(-1)} = \sum_{\ell=0}^{n}\binom{n}{\ell} \widehat{E}_{\ell}^{(0)}4^{n-\ell}
	\end{align*}
\end{corollary}
\begin{proof}
	The poly-Euler number $\widetilde{E}_{n}^{(-1)}$ counts odd $E$-sequences with only one \blue element. 
	Hence, there is at most one ordinary pair. The term $\widetilde{E}_{\ell}^{(0)}$ gives the number of constructing $S\cup R^*$. We determine the sign and the position of each \red element in an ordinary pair in $4^{n-\ell}$ ways. 
\end{proof}

Using the results of \cref{lemma:pE_special_case}, we have \cref{corollary:pE_pBC}. 

\begin{corollary}\label{corollary:pE_pBC}
	For $n,k \geq 0$, it holds 
	\begin{align*}
		\widetilde{E}_{n}^{(-k)} = \sum_{\ell=0}^{n} \binom{n}{\ell} \frac{3^\ell-1}{2}C_{n-\ell}^{(-k)} 4^{n-\ell} \quad \mbox{and}\quad
		\widehat{E}_{n}^{(-k)} = \sum_{\ell=0}^{n} \binom{n}{\ell} \frac{3^\ell+1}{2}C_{n-\ell}^{(-k)} 4^{n-\ell}.
	\end{align*}
\end{corollary}

Though the following theorem is clear from the two previous identities, we give the combinatorial explanation. 

\begin{theorem}
For $n, k \geq 0$, we have 
\begin{align*}
	\widehat{E}_n^{(-k)}-\widetilde{E}_n^{(-k)} = \sum_{\ell=0}^{n}\binom{n}{\ell}C_{n-\ell}^{(-k)}4^{n-\ell}.
\end{align*}
\end{theorem}

\begin{proof}
	We can extend the involution of the \cref{lemma:pE_special_case} on the union of the set of even and odd $E$-sequences. 
	It follows that $E$-sequences that have only one type of (say only negative signed) elements in $R^*$ and $S$ is empty can not be paired by the involution. 
	
	How many such $E$-sequences are there? We obtain such an $E$-sequence, if we choose $\ell$ \red elements for $R^*$, take a Callan sequence with $R^*$ empty (enumerated by poly-Bernoulli numbers of type $C$) and add the sign and position of the \red elements.  
\end{proof}

\cref{theo:pB_peuler} gives the relation between poly-Bernoulli numbers of type $B$ and poly-Euler numbers. 

\begin{theorem}\label{theo:pB_peuler}
For $n, k \geq 0$, it holds
\begin{align*}
	4^{n}B_n^{(-k)} = \sum_{m=0}^n \binom{n}{m} (\widetilde{E}_m^{(-k)}+ \widehat{E}_m^{(-k)}).
\end{align*}
\end{theorem}

\begin{proof}
	Take a Callan sequence and add to each \red element its sign ($+$ or $-$) and position (left or right). 
	The number of such decorated Callan sequences is $4^nB_n^{(-k)}$. Note that we separated also the extra block into two parts.

	On the other hand, choose first $m$ \red elements and construct from them an odd or an even $(m,k)$-$E$-sequence. 
	The extra \red block $R^*$ contains elements with signs. Consider these elements as being in the left part. 
	We insert elements into the right part according to the rule: Elements in $S$ will be added with a positive sign, while the remaining $n-m$ \red elements (that were not chosen for the construction of the $E$-sequence) with a negative sign. We obtain this way also all the decorated Callan sequences.  
\end{proof}

\subsection*{Acknowledgments} 

The authors would like to express their sincere gratitude to Takao Komatsu for his helpful comments.


\bibliographystyle{amsplain}
\bibliography{References1}

\end{document}